\documentclass[12pt,a4paper,oneside,reqno]{article}
\usepackage{verbatim}
\usepackage[utf8]{inputenc}
\usepackage[english]{babel}
\usepackage[left=1in,right=1in,top=1in,bottom=1in]{geometry}
\usepackage{amsfonts,amssymb,amsmath,amsthm,mathtools}
\usepackage[symbol]{footmisc}
\usepackage{enumerate}
\usepackage{cite}
\usepackage{hyperref}
\usepackage{fancyhdr}
\usepackage{fouriernc}

\hypersetup{colorlinks=true,linkcolor=black,citecolor=black,urlcolor=blue}
\theoremstyle{plain}
\newtheorem{theorem}{Theorem}
\newtheorem{lemma}[theorem]{Lemma}
\newtheorem{proposition}[theorem]{Proposition}

\theoremstyle{definition}
\newtheorem{definition}[theorem]{Definition}
\newtheorem{remark}[theorem]{Remark}
\newtheorem{example}[theorem]{Example}
\theoremstyle{plain}
\newtheorem*{theoremA}{Theorem A}
\newtheorem*{theoremB}{Theorem B}
\newtheorem*{conjectureGS}{Conjecture (Grishkov--Shestakov)}
\numberwithin{equation}{section}

\newcommand{\C}{\mathbb C}
\newcommand{\Cs}{\mathbb C\llbracket s\rrbracket}
\newcommand{\Or}{\mathcal O}
\newcommand{\vp}{\varphi}
\newcommand{\lam}{\lambda}
\newcommand{\Lam}{\Lambda}
\newcommand{\SBL}[1]{\mathrm{SBL}(#1)}
\makeatletter
\@ifundefined{llbracket}{\usepackage{stmaryrd}}{}
\makeatother

\begin{document}

\bigskip

\noindent{\Large Simple finite-dimensional binary Lie superalgebras}%
\footnote{%
The work is supported by FCT 2023.08031.CEECIND and UID/00212/2025.}%
\footnote{%
The original concept and the main results were developed by H.A.\ and I.K.
The AI system Aristotle (Harmonic, \url{https://aristotle.harmonic.fun/})
assisted the authors in the literature search, error verification and formal
verification in Lean.}

\bigskip

\begin{center}
\textbf{Hani Abdelwahab\footnote{%
Department of Mathematics, Mansoura University, Mansoura, Egypt; \
haniamar1985@gmail.com} \& Ivan Kaygorodov\footnote{%
CMA-UBI, University of Beira Interior, Covilh\~{a}, Portugal; \
kaygorodov.ivan@gmail.com}}
\end{center}

\noindent \textbf{Abstract:} \textit{We prove that every complex
finite-dimensional simple binary Lie superalgebra with non-zero odd part is
either a simple Lie superalgebra, and hence is one of the superalgebras of
Kac's list, or the two-dimensional non-Lie superalgebra $B_2.$ This confirms
the Grishkov--Shestakov conjecture and, together with the theorems of
Grishkov and Filippov on simple binary Lie algebras, gives the answer to
Shestakov's question from the Dniester Notebook (Problem 3.100(d)) over the
field of complex numbers. }

\medskip


\medskip




\section*{Introduction}

\addcontentsline{toc}{section}{Introduction}

\noindent\textbf{Binary Lie algebras.} An algebra is called \emph{binary Lie}
if any two of its elements generate a Lie subalgebra. These algebras were
introduced by Malcev~\cite{Mal} in his study of analytic loops: the tangent
algebra of a diassociative loop is binary Lie, exactly as the tangent
algebra of a Moufang loop is a Malcev algebra. Every Lie algebra and every
Malcev algebra is binary Lie. Gainov~\cite{Gai} proved that over a field of
characteristic $\neq 2$ the variety of binary Lie algebras is defined by the
anticommutativity and the identity

\begin{center}
$\bigl((xy)y\bigr)x+\bigl((yx)x\bigr)y=0.$
\end{center}

The structure theory of finite-dimensional binary Lie algebras of
characteristic zero was developed by Grishkov~\cite{G77,G80,G88}. He proved
that a finite-dimensional simple binary Lie algebra over an algebraically
closed field of characteristic zero is a Malcev algebra~\cite{G80}; this
solved Problem~1.28 of Gainov from the Dniester Notebook~\cite{DN}. By the
theorem of Filippov~\cite{Fil} such an algebra is either a simple Lie
algebra or the seven-dimensional simple Malcev algebra. Grishkov also proved
that every finite-dimensional solvable binary Lie algebra contains a
non-zero abelian ideal~\cite{G88}, which answered his own Problem~2.35 of
the Dniester Notebook. Binary Lie algebras have been studied from many other
points of view. Binary Lie algebras satisfying the third Engel condition
were studied by Filippov~\cite{Fil08}. Gainov~\cite{Gai69} studied binary
Lie algebras of characteristic~$2$ and proved that over a field with more
than three elements their variety has a finite basis of identities; this
answered Problem~1.163 of Shirshov and Gainov from the Dniester Notebook~%
\cite{DN}. Gainov~\cite{Gai63} classified binary Lie algebras of dimension
at most four and, in particular, showed that every three-dimensional binary
Lie algebra is a Lie algebra. Binary Lie algebras of small dimensions were
further studied by Kuzmin~\cite{Kuz98}. Arenas and Shestakov~\cite{AS}
considered the speciality problem for binary Lie algebras. Several other
problems on binary Lie algebras are still open; see, for example,
Problems~2.33, 2.58, 2.106, 3.34 and~3.87 of the Dniester Notebook~\cite{DN}%
. For more recent results see~\cite{GRS22,GRS} and the references therein.

\bigskip \noindent\textbf{Binary Lie superalgebras.} Let $G=G_0\oplus G_1$
be the Grassmann algebra of a countable-dimensional vector space with its
natural $\mathbb{Z}_2$-grading. A superalgebra $B=B_0\oplus B_1$ belongs to
a variety of algebras $\mathfrak{V}$ if its Grassmann envelope $%
G(B)=G_0\otimes B_0+G_1\otimes B_1$ belongs to $\mathfrak{V}$. The simple
finite-dimensional superalgebras are known for several classical varieties;
in particular, complex simple finite-dimensional Lie superalgebras were
classified by Kac~\cite{Kac}. In the fourth edition of the Dniester
Notebook, Shestakov asked to describe the simple finite-dimensional
superalgebras for some other varieties, among them for the variety of binary
Lie algebras:

\medskip \noindent\textbf{Problem 3.100(d)} (Shestakov, \cite{DN}). \textit{%
Describe all simple finite-dimensional binary Lie superalgebras.} \medskip

Binary Lie superalgebras were studied by Grishkov, Rasskazova and Shestakov~%
\cite{GRS22,GRS}. In~\cite{GRS22} they found new examples of binary Lie
superalgebras; the simplest one is the two-dimensional superalgebra $B_2$
(Example~\ref{ex:B2}), which is simple and is not a Lie superalgebra. It is
the only known simple binary Lie superalgebra with non-zero odd part that is
not a Lie superalgebra, and this led to the following conjecture of Grishkov
and Shestakov (see~\cite{GRS}).

\begin{conjectureGS}[{see {\protect\cite[Conjecture~1.1]{GRS}}}]
Let $B=B_0\oplus B_1$ be a complex finite-dimensional simple binary Lie
superalgebra with $B_1\neq0$. Then $B$ is a Lie superalgebra or $\dim B=2$.
\end{conjectureGS}

They also proved that the even part of such a superalgebra is solvable if $B$
is not a Lie superalgebra (Theorem~\ref{thm:GRS} below). Abdelwahab,
Kaygorodov and Khudoyberdiyev~\cite{AKK26} obtained the algebraic
classification of three-dimensional binary Lie superalgebras and proved the
conjecture for three-dimensional binary Lie superalgebras.

\bigskip \noindent\textbf{Our results.} We prove the Grishkov--Shestakov
conjecture. Our main result is the following theorem.

\begin{theoremA}
Let $B=B_0\oplus B_1$ be a complex finite-dimensional simple binary Lie
superalgebra with non-zero odd part $B_1$. Then $B$ is isomorphic to one of
the following superalgebras:

\begin{enumerate}
\item 
\lbrack \textrm{(1)] }

\item \textrm{a simple finite-dimensional Lie superalgebra with non-zero odd
part; these superalgebras are classified by Kac~\cite{Kac}; }

\item \textrm{the two-dimensional superalgebra $B_2$ of Example~\ref{ex:B2}. 
}
\end{enumerate}
\end{theoremA}

Together with the results of Grishkov and Filippov on the purely even case
we obtain the answer to Problem~3.100(d) in the complex case.

\begin{theoremB}
Every complex finite-dimensional simple binary Lie superalgebra is
isomorphic to a simple finite-dimensional Lie superalgebra (in particular,
to a simple Lie algebra if the odd part is zero), to the seven-dimensional
simple Malcev algebra, or to the superalgebra $B_2$.
\end{theoremB}

\newpage

\section{Preliminaries}

\label{sec:prel}

Throughout the paper $R$ is a commutative associative ring with identity in
which $2$ is invertible; for instance, $R=\mathbb{C}$ or $R$ is a
commutative associative $\mathbb{C}$-algebra with identity, such as $\mathbb{%
C}\llbracket s\rrbracket$ or $\mathbb{C}\llbracket s\rrbracket/(s^k)$. Thus
all rings of scalars are associative, while the superalgebras over them are,
in general, not associative. A \emph{superalgebra over $R$} is an $R$-module 
$S=S_0\oplus S_1$ with an $R$-bilinear multiplication such that $%
S_iS_j\subseteq S_{i+j}$ (indices modulo $2$). Elements of $S_0\cup S_1$ are 
\emph{homogeneous}, and $\bar x\in\{0,1\}$ denotes the parity of a
homogeneous element $x$.

\begin{definition}[see~\protect\cite{GRS}]
\label{def:SBL} A superalgebra $S=S_0\oplus S_1$ over $R$ is a \emph{binary
Lie superalgebra} if for all homogeneous $x,y,z,t\in S$ 
\begin{equation}  \label{eq:anti}
xy=-(-1)^{\bar x\bar y}yx
\end{equation}
and 
\begin{multline}  \label{eq:SBL}
\mathrm{SBL}(x,y,z,t):=\bigl((xy)z\bigr)t-x\bigl(y(zt)\bigr) +(-1)^{\bar
x\bar y}\Bigl\{y\bigl((xz)t\bigr)+y\bigl(x(zt)\bigr)-\bigl(y(xz)\bigr)t%
\Bigr\} \\
+(-1)^{\bar z\bar t}\Bigl\{x\bigl((yt)z\bigr)-\bigl((xy)t\bigr)z-\bigl(x(yt)%
\bigr)z\Bigr\}=0 .
\end{multline}
\end{definition}

If all arguments are even, all signs in~\eqref{eq:SBL} are $+1$. A complex
anticommutative algebra is binary Lie if and only if it satisfies~%
\eqref{eq:SBL} with all signs $+1$~\cite{Gai,GRS}, and for complex
superalgebras Definition~\ref{def:SBL} is equivalent to the definition
through the Grassmann envelope~\cite{GRS}. In a binary Lie superalgebra the
identity~\eqref{eq:anti} reads 
\begin{equation}  \label{eq:rules}
ab=-ba,\qquad xa=-ax,\qquad xy=yx\qquad (a,b\in S_0,\ x,y\in S_1);
\end{equation}
we always write a product of an even and an odd element with the even factor
first.

A complex superalgebra $S$ is \emph{simple} if $S^2\neq0$ and $0$ and $S$
are its only graded ideals. A \emph{Lie superalgebra} is a superalgebra
satisfying~\eqref{eq:anti} and the super Jacobi identity

\begin{center}
$(-1)^{\bar x\bar z}(xy)z+(-1)^{\bar y\bar x}(yz)x+(-1)^{\bar z\bar y}(zx)y=0
$.
\end{center}

\begin{example}[see \protect\cite{GRS22}]
\label{ex:B2} Let $B_2=\mathbb{C }e\oplus\mathbb{C }x$ with $e$ even, $x$
odd and 
\begin{equation*}
ee=0,\qquad ex=-xe=x,\qquad xx=e . 
\end{equation*}
This is a binary Lie superalgebra~\cite{AS,GRS22}. It is simple: $e=xx$ and $%
x=ex$, so $B_2^2=B_2$, and a non-zero graded ideal contains $e$ or $x$,
hence both. It is not a Lie superalgebra, since the super Jacobi identity
for $(x,x,x)$ gives $-3(xx)x=-3x\neq0$.
\end{example}

The following criterion is immediate from~\eqref{eq:rules}: the product $SI$
is spanned by $S_0I_0$, $S_0I_1$, $I_0S_1$ and $S_1I_1$, and $IS=SI$.

\begin{lemma}
\label{lem:ideal} Let $S=S_0\oplus S_1$ be a binary Lie superalgebra over $R$
and let $I_0\subseteq S_0$, $I_1\subseteq S_1$ be submodules. Then $%
I_0\oplus I_1$ is a graded ideal of $S$ if and only if $S_0I_0\subseteq I_0$%
, $S_0I_1\subseteq I_1$, $I_0S_1\subseteq I_1$ and $S_1I_1\subseteq I_0$.
\end{lemma}

\medskip \noindent\textbf{Grishkov's theorems.} Let $L$ be a complex binary
Lie algebra and let $V$ be a complex vector space with a bilinear map $%
L\times V\to V$, $(g,v)\mapsto gv$. The \emph{split null extension} $%
L\ltimes V$ is the space $L\oplus V$ with the multiplication $%
(g+v)(g^{\prime }+v^{\prime })=gg^{\prime }+(gv^{\prime }-g^{\prime }v)$.
The space $V$ is a \emph{binary Lie $L$-module} if $L\ltimes V$ is a binary
Lie algebra. If $U$ is a submodule of $V$, then $U$ is an ideal of $L\ltimes
V$ and $(L\ltimes V)/U\cong L\ltimes(V/U)$; since quotients of binary Lie
algebras are binary Lie, $V/U$ is again a binary Lie $L$-module.

\begin{lemma}
\label{lem:module} Let $B=A\oplus M$ be a complex binary Lie superalgebra.
Then $A$ is a binary Lie algebra and $M$ is a binary Lie $A$-module with
respect to $(a,x)\mapsto ax$.
\end{lemma}

A chain of ideals $0=J_0\subseteq J_1\subseteq\dots\subseteq J_r=L$ of an
algebra $L$ is a \emph{solvable normal series} if $J_{k+1}^2\subseteq J_k$
for all $k$.

\begin{theorem}[Grishkov]
\label{thm:G} Let $L$ be a complex finite-dimensional solvable binary Lie
algebra and let $V$ be a finite-dimensional binary Lie $L$-module.

\begin{enumerate}
\item 
\lbrack \textrm{(a)] }

\item \textrm{\cite[Theorem~A]{G80} The elements of the ideal $L^2$ act on $V
$ by nilpotent operators. }

\item \textrm{\cite[Theorem~1]{G88} If $L\neq0$, then $L$ has a non-zero
abelian ideal. }

\item \textrm{\cite[Theorem~3]{G77} If $L$ has a solvable normal series,
then every irreducible binary Lie $L$-module is one-dimensional. }
\end{enumerate}
\end{theorem}

\begin{lemma}
\label{lem:normal} A complex finite-dimensional solvable binary Lie algebra $%
L$ has a solvable normal series. Consequently, every irreducible
finite-dimensional binary Lie $L$-module is one-dimensional.
\end{lemma}

\begin{proof}
We use induction on $\dim L$; the case $L=0$ is clear. Let $L\neq0$, let $I$ be a non-zero
abelian ideal of $L$ (Theorem~\ref{thm:G}(b)) and let $\pi\colon L\to L/I$ be the canonical map.
By induction, $L/I$ has a solvable normal series $0=\bar J_0\subseteq\dots\subseteq\bar J_r=L/I$.
Put $J_0=0$ and $J_{k+1}=\pi^{-1}(\bar J_k)$, $0\le k\le r$. Then $J_1^2=I^2=0$ and
$\pi(J_{k+1}^2)\subseteq\bar J_k^2\subseteq\bar J_{k-1}$, that is, $J_{k+1}^2\subseteq J_k$ for
$k\ge1$. The second statement follows from Theorem~\ref{thm:G}(c).
\end{proof}

\section{The weight formula}

\label{sec:wf}

\subsection{Weight functionals}

Let $S=S_0\oplus S_1$ be a binary Lie superalgebra over $R$ (not necessarily
simple and not necessarily free as an $R$-module). An $R$-linear map $%
g\colon N\to R$ is surjective if and only if $g(v)=1$ for some $v\in N$; for 
$R=\mathbb{C}$ this means $g\neq0$.

\begin{definition}
\label{def:wf} A \emph{weight functional} of $S$ is a pair of surjective $R$%
-linear maps $\lambda\colon S_0\to R$ and $\varphi\colon S_1\to R$ such that 
\begin{equation}  \label{eq:wf}
\varphi(ax)=\lambda(a)\varphi(x)\qquad\text{and}\qquad\lambda(ab)=0\qquad
(a,b\in S_0,\ x\in S_1).
\end{equation}
\end{definition}

Thus, for a complex binary Lie superalgebra $B=A\oplus M$, a weight
functional is a pair of non-zero linear maps $\lambda\colon A\to\mathbb{C}$
and $\varphi\colon M\to\mathbb{C}$ satisfying~\eqref{eq:wf}. For $B_2$ the
pair $\lambda(e)=1$, $\varphi(x)=1$ is a weight functional.

\begin{proposition}
\label{prop:exist} Let $B=A\oplus M$ be a complex finite-dimensional simple
binary Lie superalgebra with $M\neq0$ and solvable even part $A$. Then $B$
has a weight functional.
\end{proposition}

\begin{proof}
By Lemma~\ref{lem:module}, $M$ is a binary Lie $A$-module. Let $W$ be a proper submodule of $M$ of
maximal dimension. Then $M/W$ is irreducible, hence $M/W=\C\bar w$ is one-dimensional by
Lemma~\ref{lem:normal}. Define $\lam$ and $\vp$ by $a\bar w=\lam(a)\bar w$ and
$x+W=\vp(x)\bar w$. Then $\ker\vp=W\neq M$ and $\vp(ax)=\lam(a)\vp(x)$. An element $e\in A^2$
acts on $M/W$ nilpotently by Theorem~\ref{thm:G}(a) and as the scalar $\lam(e)$; hence
$\lam(A^2)=0$.

It remains to show that $\lam\neq0$. Suppose $\lam=0$ and put $H=\ker\vp$. Then $AM\subseteq H$,
and Lemma~\ref{lem:ideal} shows that $A\oplus H$ is a graded ideal of $B$
($AA\subseteq A$, $AH\subseteq H$, $AM\subseteq H$, $MH\subseteq A$). It is proper, since
$\vp\neq0$, so $A\oplus H=0$ by simplicity. Then $B=M$ and $B^2\subseteq A=0$, a contradiction.
\end{proof}

\subsection{The weight formula over a commutative ring}

In this subsection $S=S_0\oplus S_1$ is a binary Lie superalgebra over $R$
(not necessarily simple and not necessarily free as an $R$-module). Given a
weight functional $(\lambda,\varphi)$ of $S$ and an element $u\in S_1$ with $%
\varphi(u)=1$, we put, for $a\in S_0$, 
\begin{equation*}
\sigma(a)=\lambda\bigl((au)u\bigr)-\lambda(a)\lambda(uu),\qquad
h_a=au-\lambda(a)u,\qquad H=\ker\varphi . 
\end{equation*}
By~\eqref{eq:rules}, $\lambda(xy)=\lambda(yx)$ for $x,y\in S_1$, and $\sigma$
is $R$-linear. Every $x\in S_1$ decomposes as $x=\varphi(x)u+\bigl(%
x-\varphi(x)u\bigr)$ with $x-\varphi(x)u\in H$, so $S_1=Ru\oplus H$.
Moreover $S_0H\subseteq H$, $h_a\in H$ and $\sigma(a)=\lambda({h_a}u)$.

\begin{lemma}
\label{lem:scalar} Let $(\lambda,\varphi)$ be a weight functional of $S$.
Then for all $a,b\in S_0$ and $y,z,t\in S_1$ 
\begin{gather}
\Bigl(\lambda\bigl((ay)z\bigr)-\lambda\bigl((az)y\bigr)\Bigr)%
\varphi(t)+\lambda\bigl((ay)t\bigr)\varphi(z)-\lambda\bigl((az)t\bigr)%
\varphi(y) +\lambda(a)\bigl(\lambda(zt)\varphi(y)-\lambda(yt)\varphi(z)\bigr)%
=0,  \label{eq:phi-I} \\
\lambda\Bigl(\bigl((ab)z\bigr)t\Bigr)-\lambda\Bigl(\bigl(b(az)\bigr)t\Bigr)%
+\lambda\Bigl(\bigl((ab)t\bigr)z\Bigr)+\lambda\Bigl(\bigl(a(bt)\bigr)z\Bigr)%
=0.  \label{eq:lam-half}
\end{gather}
\end{lemma}

\begin{proof}
By~\eqref{eq:rules}, the identities $\SBL{a,y,z,t}=0$ and $\SBL{a,b,z,t}=0$ can be written as
\begin{gather*}
\bigl((ay)z\bigr)t+a\bigl((zt)y\bigr)-\bigl((az)t\bigr)y-\bigl(a(zt)\bigr)y-\bigl((az)y\bigr)t-a\bigl((yt)z\bigr)+\bigl((ay)t\bigr)z+\bigl(a(yt)\bigr)z=0,\\
\bigl((ab)z\bigr)t-a\bigl(b(zt)\bigr)+b\bigl((az)t\bigr)+b\bigl(a(zt)\bigr)-\bigl(b(az)\bigr)t-a\bigl((bt)z\bigr)+\bigl((ab)t\bigr)z+\bigl(a(bt)\bigr)z=0.
\end{gather*}
Recall that $\vp(em)=\lam(e)\vp(m)$, $mm'=m'm$ and $\lam(S_0S_0)=0$ for $e\in S_0$
and $m,m'\in S_1$. Applying $\vp$ to the first identity, the eight terms give, in order,
\begin{multline*}
\lam\bigl((ay)z\bigr)\vp(t),\ \ 
\lam(a)\lam(zt)\vp(y),\ \ 
-\lam\bigl((az)t\bigr)\vp(y),\ \ 
0,\\
-\lam\bigl((az)y\bigr)\vp(t),\ \ -\lam(a)\lam(yt)\vp(z),\ \ \lam\bigl((ay)t\bigr)\vp(z),\ \ 0,
\end{multline*}
\noindent and their sum is~\eqref{eq:phi-I}. Applying $\lam$ to the second identity, the second, third,
fourth and sixth terms lie in $S_0S_0$ and vanish, and the remaining four terms
give~\eqref{eq:lam-half}.
\end{proof}

\begin{lemma}
\label{lem:basic2} Let $(\lambda,\varphi)$ be a weight functional of $S$ and 
$\varphi(u)=1$. Then for all $a,b\in S_0$, $p,q\in H$ and $z,t\in S_1$ 
\begin{gather}
\lambda\bigl((ap)q\bigr)=\lambda(a)\lambda(pq),\qquad 2\lambda\bigl((ap)u%
\bigr)=\lambda\bigl((au)p\bigr)+\lambda(a)\lambda(pu),  \label{eq:basic} \\
\lambda\Bigl(\bigl(a(bz)+b(az)\bigr)t\Bigr)=\lambda\Bigl(z\bigl(a(bt)+b(at)%
\bigr)\Bigr),  \label{eq:lam-I} \\
\lambda({h_a}p)=0.  \label{eq:hp}
\end{gather}
\end{lemma}

\begin{proof}
The identities~\eqref{eq:basic} are~\eqref{eq:phi-I} with $(y,z,t)=(u,p,q)$ and
$(y,z,t)=(u,p,u)$. Interchanging $a$ and $b$ in~\eqref{eq:lam-half}, using $ba=-ab$, adding the
result to~\eqref{eq:lam-half} and using $\lam(xy)=\lam(yx)$ for $x,y\in S_1$, we get~\eqref{eq:lam-I}.

To prove~\eqref{eq:hp}, fix $p\in H$ and put $\ell=\lam(pu)$ and $\varepsilon(a)=\lam(p{h_a})$.
Since $au=\lam(a)u+h_a$, we have $\lam\bigl((au)p\bigr)=\lam(a)\ell+\varepsilon(a)$. We compute both sides
of~\eqref{eq:lam-I} for $z=p$, $t=u$. By~\eqref{eq:basic},
$2\lam\bigl((bp)u\bigr)=2\lam(b)\ell+\varepsilon(b)$. Since $bp,h_a\in H$, \eqref{eq:basic} applied to $bp$
gives
\[
2\lam\Bigl(\bigl(a(bp)\bigr)u\Bigr)=2\lam(a)\lam\bigl((bp)u\bigr)+\lam\bigl((bp){h_a}\bigr)=2\lam(a)\lam(b)\ell+\lam(a)\varepsilon(b)+\lam(b)\varepsilon(a).
\]
This expression is symmetric in $a,b$, so the left side of~\eqref{eq:lam-I} equals
\begin{center}
$2\lam(a)\lam(b)\ell+\lam(a)\varepsilon(b)+\lam(b)\varepsilon(a)$.
\end{center}  On the other hand,
\begin{center}$a(bu)=\lam(b)au+ah_b$, and
$\lam\Bigl(p\bigl(a(bu)\bigr)\Bigr)=\lam(a)\lam(b)\ell+\lam(b)\varepsilon(a)+\lam(a)\varepsilon(b)$,\end{center} so the right side
of~\eqref{eq:lam-I} equals
$2\lam(a)\lam(b)\ell+2\lam(a)\varepsilon(b)+2\lam(b)\varepsilon(a)$. Hence
\begin{center}
$\lam(a)\varepsilon(b)+\lam(b)\varepsilon(a)=0$ for all $a,b\in S_0$.\end{center} Choose $h\in S_0$ with
$\lam(h)=1$. For $a=b=h$ we get $2\varepsilon(h)=0$, so $\varepsilon(h)=0$; then $a=h$ gives
$\varepsilon(b)=0$ for all $b$.
\end{proof}

\begin{theorem}[The weight formula]
\label{thm:WF} Let $(\lambda,\varphi)$ be a weight functional of a binary
Lie superalgebra $S$ over $R$, and let $\varphi(u)=1$. Then 
\begin{equation}  \label{eq:WF}
\lambda\bigl((ax)y\bigr)=\lambda(a)\lambda(xy)+\sigma(a)\varphi(x)\varphi(y)%
\qquad (a\in S_0,\ x,y\in S_1)
\end{equation}
and $\sigma(S_0S_0)=0$. In particular, $\sigma$ does not depend on the
choice of $u$.
\end{theorem}

\begin{proof}
Let $a\in S_0$ and $p,q\in H$. By~\eqref{eq:basic}, $\lam\bigl((ap)q\bigr)=\lam(a)\lam(pq)$. Since
$au=\lam(a)u+h_a$, \eqref{eq:hp} gives $\lam\bigl((au)q\bigr)=\lam(a)\lam(uq)$, and then the second identity
in~\eqref{eq:basic} gives $$\lam\bigl((ap)u\bigr)=\lam(a)\lam(pu).$$ Finally,
$\lam\bigl((au)u\bigr)=\lam(a)\lam(uu)+\sigma(a)$. Writing $x=\alpha u+p$, $y=\beta u+q$ with
$\alpha=\vp(x)$, $\beta=\vp(y)$ and $p,q\in H$, and expanding $\lam\bigl((ax)y\bigr)$ by bilinearity, we
obtain~\eqref{eq:WF}.

Let $a,b\in S_0$. Putting $z=t=u$ in~\eqref{eq:lam-half} we get
$$2\lam\Bigl(\bigl((ab)u\bigr)u\Bigr)-\lam\Bigl(\bigl(b(au)\bigr)u\Bigr)+\lam\Bigl(\bigl(a(bu)\bigr)u\Bigr)=0.$$ Since $\lam(ab)=0$, \eqref{eq:WF} gives
$\lam\Bigl(\bigl((ab)u\bigr)u\Bigr)=\sigma(ab)$; since $\vp(au)=\lam(a)$, it also gives
$$\lam\Bigl(\bigl(b(au)\bigr)u\Bigr)=\lam(a)\lam(b)\lam(uu)+\lam(b)\sigma(a)+\lam(a)\sigma(b),$$ which is symmetric in
$a,b$. Hence $2\sigma(ab)=0$ and $\sigma(ab)=0$. If $\vp(u')=1$, then~\eqref{eq:WF} with
$x=y=u'$ gives $\sigma(a)=\lam\bigl((au')u'\bigr)-\lam(a)\lam(u'u')$, so $\sigma$ does not depend on $u$.
\end{proof}

\subsection{The functional \texorpdfstring{$\sigma$}{sigma} vanishes}

In this subsection $B=A\oplus M$ is a complex finite-dimensional binary Lie
superalgebra (not necessarily simple) with a weight functional $%
(\lambda,\varphi)$.

We use the power series ring $\mathbb{C}\llbracket s\rrbracket$ and its
quotients $\mathcal{O}_k=\mathbb{C}\llbracket s\rrbracket/(s^k)$, $k\ge1$.
For $f\in\mathbb{C}\llbracket s\rrbracket$, $f(0)$ is the constant term, $%
\bar f$ is the image of $f$ in $\mathcal{O}_k$ and $\partial f$ is the
formal derivative; $f\equiv g\pmod{s^k}$ means $f-g\in s^k\mathbb{C}%
\llbracket s\rrbracket$. We use the standard facts: $f$ is invertible in $%
\mathbb{C}\llbracket s\rrbracket$ (or in $\mathcal{O}_k$) if and only if $%
f(0)\neq0$; $\mathbb{C}\llbracket s\rrbracket$ is an integral domain; if $%
f(0)=0$ and $\partial f\in s^k\mathbb{C}\llbracket s\rrbracket$, then $f\in
s^{k+1}\mathbb{C}\llbracket s\rrbracket$; and $\bigcap_k s^k\mathbb{C}%
\llbracket s\rrbracket=0$.

Fix $u\in M$ with $\varphi(u)=1$ and let $\Theta\colon M\to M$, $%
\Theta(z)=(uz)u$. Define $\mathbb{C}$-linear maps $\Psi,\Phi\colon M\to%
\mathbb{C}\llbracket s\rrbracket$ and $\Lambda\colon A\to\mathbb{C}%
\llbracket s\rrbracket$ by 
\begin{equation*}
\Psi(z)=\sum_{j\ge0}\varphi\bigl(\Theta^j(z)\bigr)\,\frac{s^j}{j!},\qquad
\Phi(z)=\Psi(u)^{-1}\Psi(z),\qquad \Lambda(a)=\Phi(au), 
\end{equation*}
so that $\Phi$ is ``$\varphi\circ e^{s\Theta}$ normalized at $u$''; $\Psi(u)$
is invertible since its constant term is $\varphi(u)=1$, and $\Phi(u)=1$.
For $a\in A$ put 
\begin{equation*}
\tau(a)=\Lambda\bigl((au)u\bigr)-\Lambda(a)\Lambda(uu). 
\end{equation*}
Note that $\Lambda(xy)=\Lambda(yx)$ for $x,y\in M$. We will show that the
pair $(\Lambda,\Phi)$ satisfies the identities~\eqref{eq:wf}, that is, $%
\Phi(az)=\Lambda(a)\Phi(z)$ and $\Lambda(ab)=0$ for all $a,b\in A$ and $z\in
M$.

\begin{lemma}
\label{lem:Psi} For all $a,b\in A$ and $x,y,z\in M$:

\begin{enumerate}
\item 
\lbrack \textrm{(a)] }

\item \textrm{$\Phi(z)(0)=\varphi(z)$, $\Lambda(a)(0)=\lambda(a)$, $%
\Lambda(xy)(0)=\lambda(xy)$ and $\tau(a)(0)=\sigma(a)$; }

\item \textrm{$\partial\Phi(z)=\Lambda(uz)-\Phi(z)\Lambda(uu)$; }

\item \textrm{$\partial\Lambda(a)=\tau(a)$ and $\partial\Lambda(ab)=\tau(ab)$%
; }

\item \textrm{$\partial\bigl(\Phi(az)-\Lambda(a)\Phi(z)\bigr)=\Lambda\bigl(%
(az)u\bigr)-\Lambda(a)\Lambda(zu)-\tau(a)\Phi(z)-\Lambda(uu)\bigl(%
\Phi(az)-\Lambda(a)\Phi(z)\bigr)$. }
\end{enumerate}
\end{lemma}

\begin{proof}
(a) The constant term of $\Psi(z)$ is $\vp(z)$; taking constant terms in
$\Psi(z)=\Psi(u)\Phi(z)$ gives $\Phi(z)(0)=\vp(z)$. Hence $\Lam(a)(0)=\vp(au)=\lam(a)$,
$\Lam(xy)(0)=\lam(xy)$ and $\tau(a)(0)=\sigma(a)$.

(b) We have $\partial\Psi(z)=\Psi\bigl(\Theta(z)\bigr)$. Differentiating $\Psi(z)=\Psi(u)\Phi(z)$ and dividing
by $\Psi(u)$, we get $\Phi\bigl(\Theta(z)\bigr)=\Phi\bigl(\Theta(u)\bigr)\Phi(z)+\partial\Phi(z)$. Since
$\Theta(w)=(uw)u$ with $uw\in A$, we have $\Phi\bigl(\Theta(w)\bigr)=\Lam(uw)$, and (b) follows.

(c) Apply (b) to $z=au$ and use $\Phi(au)=\Lam(a)$ and $u(au)=(au)u$. The second
formula follows since $ab\in A$.

(d) By (b), $\partial\Phi(az)=\Lam\bigl((az)u\bigr)-\Lam(uu)\Phi(az)$; by the product rule, (b) and (c),
$\partial\bigl(\Lam(a)\Phi(z)\bigr)=\tau(a)\Phi(z)+\Lam(a)\Lam(zu)-\Lam(uu)\Lam(a)\Phi(z)$. Subtract.
\end{proof}

\medskip \noindent\textbf{Extension of scalars.} Let $\mathcal{O}$ be a
commutative associative $\mathbb{C}$-algebra with identity. Then $B_\mathcal{%
O}=B\otimes_\mathbb{C}\mathcal{O}=A_\mathcal{O}\oplus M_\mathcal{O}$, with $%
A_\mathcal{O}=A\otimes_\mathbb{C}\mathcal{O}$, $M_\mathcal{O}=M\otimes_%
\mathbb{C}\mathcal{O}$, is a superalgebra over $\mathcal{O}$ with the
multiplication $(b\otimes f)(c\otimes g)=bc\otimes fg$. If $e_1,\dots,e_n$
is a homogeneous basis of $B$, then $B_\mathcal{O}$ is the free $\mathcal{O}$%
-module with the same basis and the same multiplication table. Both sides of~%
\eqref{eq:anti} and~\eqref{eq:SBL} are multilinear, and $\mathrm{SBL}%
(b_1\otimes f_1,\dots,b_4\otimes f_4)=\mathrm{SBL}(b_1,\dots,b_4)\otimes
f_1f_2f_3f_4$ since $\mathcal{O}$ is commutative. Hence we have the
following statement.

\begin{proposition}
\label{prop:BO} $B_\mathcal{O}$ is a binary Lie superalgebra over $\mathcal{O%
}$ with even part $A_\mathcal{O}$ and odd part $M_\mathcal{O}$.
\end{proposition}

\begin{lemma}
\label{lem:trunc} Let $k\ge1$ and assume that $\Phi(az)\equiv\Lambda(a)%
\Phi(z)$ and $\Lambda(ab)\equiv0\pmod{s^k}$ for all $a,b\in A$ and $z\in M$.
Then for all $a,b\in A$ and $x,y\in M$ 
\begin{equation}  \label{eq:WF-k}
\Lambda\bigl((ax)y\bigr)\equiv\Lambda(a)\Lambda(xy)+\tau(a)\Phi(x)\Phi(y)%
\qquad\text{and}\qquad \tau(ab)\equiv0\pmod{s^k}.
\end{equation}
\end{lemma}

\begin{proof}
By Proposition~\ref{prop:BO}, $B_{\Or_k}$ is a binary Lie superalgebra over $\Or_k$. Let
$\lam_k\colon A_{\Or_k}\to\Or_k$ and $\vp_k\colon M_{\Or_k}\to\Or_k$ be the $\Or_k$-linear
extensions of $a\mapsto\overline{\Lam(a)}$ and $z\mapsto\overline{\Phi(z)}$ (they exist and are
unique since $B_{\Or_k}$ is free over $\Or_k$ with a basis of $B$). The hypothesis says exactly
that~\eqref{eq:wf} holds for basis elements, hence, by bilinearity, for all elements. Moreover,
$\vp_k(u)=1$, and if $\lam(h)=1$, then $\lam_k(h)$ has constant term $1$ and is invertible. So
$(\lam_k,\vp_k)$ is a weight functional of $B_{\Or_k}$. For this pair we have
$\lam_k(xy)=\overline{\Lam(xy)}$ and $\sigma_k(a)=\overline{\tau(a)}$ for $a\in A$, $x,y\in M$,
where $\sigma_k$ is the functional $\sigma$ defined for $(\lam_k,\vp_k)$ and $u$.
Theorem~\ref{thm:WF} for $B_{\Or_k}$ gives~\eqref{eq:WF-k}.
\end{proof}

\begin{theorem}
\label{thm:sigma} Let $B=A\oplus M$ be a complex finite-dimensional binary
Lie superalgebra and let $(\lambda,\varphi)$ be a weight functional of $B$.
Then $\sigma=0$, that is, 
\begin{equation}  \label{eq:inv}
\lambda\bigl((ax)y\bigr)=\lambda(a)\lambda(xy)\qquad (a\in A,\ x,y\in M).
\end{equation}
\end{theorem}

\begin{proof} Our proof is based on the following three steps.
\begin{enumerate}
 \item[{\bf Step 1:}] We prove that $\Phi(az)=\Lam(a)\Phi(z)$ and $\Lam(ab)=0$ for all $a,b\in A$ and $z\in M$. 
 
 It
suffices to prove $\Phi(az)\equiv\Lam(a)\Phi(z)$ and $\Lam(ab)\equiv0\pmod{s^k}$ for all $k\ge1$;
we use induction on $k$. For $k=1$ the constant terms are $\vp(az)-\lam(a)\vp(z)=0$ and
$\lam(ab)=0$ by Lemma~\ref{lem:Psi}(a). Assume the congruences modulo $s^k$. By
Lemma~\ref{lem:trunc} with $x=z$, $y=u$ and $\Phi(u)=1$,
\[
\Lam\bigl((az)u\bigr)-\Lam(a)\Lam(zu)-\tau(a)\Phi(z)\in s^k\Cs ,
\]
so $\partial\bigl(\Phi(az)-\Lam(a)\Phi(z)\bigr)\in s^k\Cs$ by Lemma~\ref{lem:Psi}(d); also
$\partial\Lam(ab)=\tau(ab)\in s^k\Cs$ by Lemmas~\ref{lem:Psi}(c) and~\ref{lem:trunc}. Since
$\Phi(az)-\Lam(a)\Phi(z)$ and $\Lam(ab)$ have zero constant terms,
$\Phi(az)-\Lam(a)\Phi(z)\in s^{k+1}\Cs$ and $\Lam(ab)\in s^{k+1}\Cs$.

\item[{\bf Step 2:}] We prove that $\Lam(a)$ is a constant. 

Fix $a\in A$ and let $\rho\colon M\to M$, $\rho(z)=az$. By
Step~1, $\Phi\bigl(\rho^m(z)\bigr)=\Lam(a)^m\Phi(z)$ for all $m\ge0$, hence
$\Phi\bigl(p(\rho)(z)\bigr)=p\bigl(\Lam(a)\bigr)\Phi(z)$ for every polynomial $p\in\C[t]$. Let
$\chi(t)=\prod_{i=1}^n(t-\alpha_i)$ be the characteristic polynomial of $\rho$
($n=\dim M\ge1$). By the Cayley--Hamilton theorem,
$0=\Phi\bigl(\chi(\rho)(u)\bigr)=\prod_i\bigl(\Lam(a)-\alpha_i\bigr)$. Since $\Cs$ is a domain, $\Lam(a)=\alpha_i$ for
some $i$.

\item[{\bf Step 3:}] By Step~2 and Lemma~\ref{lem:Psi}(c), $\tau(a)=\partial\Lam(a)=0$, and by
Lemma~\ref{lem:Psi}(a), $\sigma(a)=\tau(a)(0)=0$. Now~\eqref{eq:WF} becomes~\eqref{eq:inv}.
\end{enumerate}\end{proof}

\section{Proof of Theorems A and B}

\label{sec:proof}

In this section $B=A\oplus M$ is a complex finite-dimensional binary Lie
superalgebra with a weight functional $(\lambda,\varphi)$.

\begin{lemma}
\label{lem:rank1} $\lambda(xy)\lambda(zt)=\lambda(xz)\lambda(yt)$ for all $%
x,y,z,t\in M$.
\end{lemma}

\begin{proof}
For $x',y',z',t'\in M$ we have $x'y'\in A$, so Theorem~\ref{thm:sigma} gives
$$\lam\Bigl(\bigl((x'y')z'\bigr)t'\Bigr)=\lam(x'y')\lam(z't').$$ 
By~\eqref{eq:rules}, the eight terms of
$\SBL{x,y,z,t}$ for odd $x,y,z,t$ are
\[
\bigl((xy)z\bigr)t,\ \ \bigl((zt)y\bigr)x,\ \ -\bigl((xz)t\bigr)y,\ \ \bigl((zt)x\bigr)y,\ \ -\bigl((xz)y\bigr)t,\ \ -\bigl((yt)z\bigr)x,\ \
\bigl((xy)t\bigr)z,\ \ -\bigl((yt)x\bigr)z .
\]
Applying $\lam$ and using $mm'=m'm$ for $m,m'\in M$, we get
$0=4\lam(xy)\lam(zt)-4\lam(xz)\lam(yt)$.
\end{proof}

\begin{lemma}
\label{lem:ideal-K-R} Let $M^{\perp}=\bigl\{r\in M:\lambda(Mr)=0\bigr\}$.
Then $\ker\lambda\oplus M^{\perp}$ is a proper graded ideal of $B$.
\end{lemma}

\begin{proof}
Put $K=\ker\lam$; it is a proper subspace of $A$, since $\lam\neq0$. We check the conditions of
Lemma~\ref{lem:ideal} using~\eqref{eq:inv}: $AK\subseteq AA\subseteq K$; if $r\in M^{\perp}$, then
$\lam\bigl((ar)y\bigr)=\lam(a)\lam(ry)=0$, so $AM^{\perp}\subseteq M^{\perp}$; if $e\in K$, then
$\lam\bigl((ex)y\bigr)=\lam(e)\lam(xy)=0$, so $KM\subseteq M^{\perp}$; finally, $\lam(xr)=0$ for
$x\in M$, $r\in M^{\perp}$, so $MM^{\perp}\subseteq K$.
\end{proof}

\begin{theorem}
\label{thm:B2} If $B$ is simple, then $\dim A=\dim M=1$ and $B\cong B_2$.
\end{theorem}

\begin{proof}
By Lemma~\ref{lem:ideal-K-R} and simplicity, $\ker\lam=0$ and $M^{\perp}=0$, i.e., the symmetric
bilinear form $(x,y)\mapsto\lam(xy)$ on $M$ is non-degenerate. If $\dim M\ge2$, then, since this is a
non-degenerate symmetric bilinear form on a complex vector space, there are $x,z\in M$ with $\lam(xx)=\lam(zz)=1$ and $\lam(xz)=0$, and Lemma~\ref{lem:rank1} for
$(x,x,z,z)$ gives $1=0$. Hence $M=\C u$ with $\vp(u)=1$, and $A=\C h$ with $\lam(h)=1$, since
$\ker\lam=0$. Now $hh=0$; $hu\in\C u$ and $\vp(hu)=\lam(h)\vp(u)=1$, so $hu=u$; and $uu=\delta h$
with $\delta=\lam(uu)\neq0$. For $\beta\in\C$ with $\beta^2\delta=1$ and $u'=\beta u$ we
get $hh=0$, $hu'=u'$, $u'u'=h$, which is the multiplication table of $B_2$.
\end{proof}

The following theorem was proved in~\cite{GRS}.

\begin{theorem}[see \protect\cite{GRS}]
\label{thm:GRS} Let $B=B_0\oplus B_1$ be a complex finite-dimensional simple
binary Lie superalgebra with $B_1\neq0$. If $B$ is not a Lie superalgebra,
then $B_0$ is solvable.
\end{theorem}

\begin{proof}[Proof of Theorem~A]
Let $B=A\oplus M$ be a complex finite-dimensional simple binary Lie superalgebra with $M\neq0$.
If $B$ is a Lie superalgebra, then it is one of the simple Lie superalgebras of Kac's
list~\cite{Kac}. Otherwise $A$ is solvable by Theorem~\ref{thm:GRS}, $B$ has a weight functional
by Proposition~\ref{prop:exist}, and $B\cong B_2$ by Theorem~\ref{thm:B2}.
\end{proof}

\begin{proof}[Proof of Theorem~B]
If $B_1\neq0$, we apply Theorem~A. If $B_1=0$, then $B$ is a complex simple finite-dimensional binary Lie
algebra; by the theorem of Grishkov~\cite{G80} it is a Malcev algebra, and by the
theorem of Filippov~\cite{Fil} it is either a simple Lie algebra or the seven-dimensional simple
Malcev algebra.
\end{proof}

\end{document}